\documentclass[10pt, oneside]{amsart}
\usepackage{amssymb,latexsym}
\usepackage{enumerate}
\usepackage{pstricks}
\usepackage{pst-node}
\newrgbcolor{colorato}{0.8  0.08  0}

\usepackage{color}
\newtheorem{thm}{Theorem}[section]
\newtheorem{cor}[thm]{Corollary}

\theoremstyle{definition}

\newtheorem{definition}[thm]{Definition}
\newtheorem{rem}[thm]{Remark}
\newtheorem{exa}[thm]{Example}

\newtheorem{proposition}[thm]{Proposition}

\numberwithin{equation}{section}
\usepackage[margin=1.5in]{geometry}

\usepackage{euscript}
\newcommand{\Tr}{T_{\!\rm rich}}

\newcommand{\E}{\exists}
\newcommand{\A}{\forall}

\newcommand{\K}{{\EuScript K}}
\newcommand{\imp}{\rightarrow}
\newcommand{\Obj}{\textrm{Obj}}
\newcommand{\Mor}{\textrm{Mor}}
\newcommand{\dom}{\textrm{dom}}
\newcommand{\range}{\textrm{rng}}
\newcommand{\IFF}{\Leftrightarrow}
\newcommand{\id}{\textrm{id}}
\newcommand{\IMP}{\Rightarrow}
\newcommand{\Th}{\mathop{\textrm{Th}}}

\newcommand{\acl}{\mathord{\textrm{acl}}}

\renewcommand{\phi}{\varphi}

\author[S. Barbina]{Silvia Barbina}
\address{Cambridge University Press\\ Shaftesbury Road, Cambridge CB2 8RU, UK}
\email{silvia.barbina@gmail.com}

\author[D. Zambella]{Domenico Zambella}
\address{Dipartimento di Matematica\\Universit\`a di Torino\\via Carlo Alberto 10, 10123 Torino, Italy}
\email{domenico.zambella@unito.it}

\begin{document}
\baselineskip=3ex

\title{A viewpoint on amalgamation classes}

\date{}

\begin{abstract}We provide a self-contained introduction to the classical theory of 
universal-homogeneous models (also known as generic structures, rich models, 
or Fra\"iss\'e limits). In the literature, most treatments restrict 
consideration to embeddings among finite structures. This is not suitable 
for some applications. We take the notion of morphisms as primitive and we 
allow structures to have arbitrary cardinality.

\end{abstract}
\maketitle

\section{Introduction}
Universal homogeneous models, here called \emph{rich} models, are a fundamental tool in model theory. They were first introduced by Fra\"{\i}ss\'e and in the last two decades they have become a basic tool for the construction of a variety of (counter)examples --- see for instance \cite{hrushovski}, \cite{carre}, \cite{bose} and many others. Rich models are usually constructed by axiomatizing the notion of strong submodel. Here we present an axiomatization based on the notion of morphism.

The concept of model companion is closely related to the notion of rich model. For instance, the random graph can be obtained as the Fra\"{\i}ss\'e limit of the class of all finite graphs, but it can also be defined as the model companion of the theory of infinite graphs. Generic automorphisms, introduced by Lascar as \textit{beaux automorphismes} in~\cite{lascar}, can be obtained either as Fra\"{\i}ss\'e limits or as model companions as in~\cite{chpi}  (see also \cite{balshel} and \cite{generici}).

The connection between these two approaches is well understood when the amalgamation class is \textit{connected}, i.e. it satisfies the joint embedding property (JEP), but the relationship is less clear when JEP fails. In Section~\ref{saturation} we produce an example of an amalgamation class where each connected component has a saturated rich model but the theory of the rich models is not model-complete (see Remark~\ref{fullnecessary}). Sections~\ref{saturation}  and~\ref{companions} are dedicated to surveying the relation between the saturation of the rich models and the model-completeness of their theory. They collect facts that to our knowledge have never been treated in a comprehensive self-contained way.

\section{Inductive amalgamation classes}\label{definitions}

In this section we present an axiomatization of {\it inductive amalgamation classes\/} based on the notion of morphism. This differs from the approach commonly found in the literature, where the primitive notion is that of strong submodel (here denoted by $\le$). 

In order to state our axioms it is essential to explain the meaning of the word \emph{map\/} in this paper. A map $f:M\imp N$ is a \emph{triple\/} where $M$ is a structure called the \emph{domain\/} of the map, $N$ is a structure called the \emph{co-domain\/} of the map, and $f$ is a function in the set-theoretic sense with $\dom f\subseteq M$ and $\range f\subseteq N$. We call $\dom f$ the \emph{domain of definition\/} of the map and $\range f$ the \emph{range\/} of the map. If $A\subseteq\dom f$ we say that $f$ is \emph{defined on\/} $A$. So $f:M\imp N$ and $f:M'\imp N'$ are different maps unless $M=M'$ and $N=N'$. 

The \emph{composition\/} of two maps is defined when the co-domain of the first map is the domain of the second map. Clearly, composing two non total maps may give the empty map as a result. When $f:M\imp N$ is injective (which will always be the case in this paper) its  \emph{inverse\/} is the map $f^{-1}:N\imp M$.

When $M$ and $N$ are structures in a given signature, a \emph{partial embedding\/} is a map $f:M\imp N$ such that $M\models\phi(a)\ \IFF\ N\models\phi(f a)$ for every quantifier-free formula $\phi(x)$ and every tuple $a \subseteq \dom f$.  An \emph{elementary map\/} is defined similarly but with $\phi(x)$ ranging over all formulas. A partial embedding which is a total map is called an \emph{embedding\/} and a total elementary map is called an \emph{elementary embedding}. 

\begin{definition}
Fix a countable language $L$. An \emph{inductive amalgamation class\/} $\K$ is a category where $\Obj(\K)$ consists of infinite structures of signature $L$, $\Mor(\K)$ contains partial embeddings between structures, and which satisfies axioms \textsf{\footnotesize K0}, \textsf{\footnotesize K1}, \textsf{\footnotesize K2}, \textsf{\footnotesize R}, \textsf{\footnotesize Ap} and \textsf{\footnotesize In} below, where composition of morphisms is composition of maps, a \emph{model\/} is an element of $\Obj(\K)$ and a \emph{morphism\/} is an element of  $\Mor(\K)$.
\begin{itemize}
\item[\textsf{\footnotesize K0}.] models are closed under elementary equivalence;
\item[\textsf{\footnotesize K1}.] all elementary maps are morphisms;
\item[\textsf{\footnotesize K2}.] the inverse (in the sense above) of a morphism is a morphism;
\item[\textsf{\footnotesize R}.]  if $h:M\imp N$ is a morphism and $f\subseteq h$ then $f:M\imp N$ is a morphism.
\end{itemize}

A morphism that is total is called a \emph{strong embedding}. The structure $M$ is a \emph{strong submodel\/} of $N$, written $M\le N$, if $M\subseteq N$ and $\id_M:M\imp N$ is a morphism (hence a strong embedding). We call $h:M'\imp N'$ an \emph{extension\/} of $f:M\imp N$ if $M\le M'$, $N\le N'$ and $f\subseteq h$. 

\begin{itemize}
\item[\textsf{\footnotesize Ap}.] Every morphism $f:M\imp N$ has an extension to a strong embedding $h:M\imp N'$.
\end{itemize}
A \emph{chain of models\/} is a sequence of models $\langle M_i : i < \lambda \rangle$ such that $M_i\le M_j$ whenever $i<j$.
\begin{itemize}
\item[\textsf{\footnotesize In}.] The union $M$ of a chain of models  $\langle M_i : i < \lambda \rangle$ is a model and $M_i\le M$ for every $i<\lambda$.
\end{itemize}
\end{definition}

In \textsf{\footnotesize K2} the word \textit{inverse\/} does not have the meaning it has in a category: the composition of $f:M\imp N$ and $f^{-1}:N\imp M$ is not $\id_M$ but merely the identity on $\dom f$. Axiom \textsf{\footnotesize R} is not essential but it is assumed to simplify the exposition. If $\K$ satisfies all the axioms above except for \textsf{\footnotesize R}, we define an inductive amalgamation class $\K'$ whose objects are those of $\K$ and whose morphisms are

\hfil $\Mor(\K')\ \ =\ \ \Big\{h:M\imp N\ \ |\ \ \Mor(\K)\textrm{ contains a restriction of }\ h:M\imp N\ \Big\}$.

For our purposes, we can safely replace $\K$ with $\K'$. Axiom \textsf{\footnotesize Ap} is a convenient way to formulate the amalgamation property. This is usually stated as in \textsf{\footnotesize Ap$'$} below.

\begin{proposition}
Modulo \textsf{\footnotesize K0}-\textsf{\footnotesize K2}, axiom \textsf{\footnotesize Ap} is equivalent to the following
\begin{itemize}
\item[\textsf{\footnotesize Ap$'$}.] if $f_i:M\imp N_i$ for $i=1,2$ are morphisms then there is a model $N$ and two strong embeddings $h_i:N_i\imp N$ such that $h_1\,f_1\mathord\restriction\dom f_2=h_2\,f_2\mathord\restriction\dom f_1$.
\end{itemize}
\end{proposition}
\begin{proof}
Observe first that if $f:M\imp N$ is a strong embedding then $f[M]\le N$. In fact, $f^{-1}:f[M]\imp M$ is an isomorphism so, in particular, an elementary map. Then, by \textsf{\footnotesize K0}, $f[M]$ is a model and by \textsf{\footnotesize K1} $f^{-1}:f[M]\imp M$ is a morphism. Composing it with $f:M\imp N$, we can conclude that the natural embedding of $f[M]$ into $N$ is a morphism.

To prove \textsf{\footnotesize Ap$'$} $\IMP$ \textsf{\footnotesize Ap}, amalgamate $f:M\imp N$ and $\id_M:M\imp M$. For the converse, apply \textsf{\footnotesize Ap} to the morphism $f_2f^{-1}_1:N_1\imp N_2$ to obtain a strong embedding $h:N_1\imp N$ into some $N_2\le N$. This and $\id_{N_2}:N_2\imp N$ are the two embeddings $h_i:N_i\imp N$ required in \textsf{\footnotesize Ap$'$}.
\end{proof}

We say that $\K$ is \emph{connected\/} if between any two models there is a morphism. The following is an immediate consequence of amalgamation.

\begin{proposition}
The following are equivalent for any amalgamation class $\K$:
\begin{itemize}
\item[\textsf{\footnotesize C}.] $\K$ is connected;
\item[\textsf{\footnotesize Jep}.] For every pair of models $M_1$ and $M_2$ there are a model $N$ and embeddings $f_i:M_i\imp N$ for $i=1,2$.
\end{itemize}
\end{proposition}

An example of an inductive amalgamation class is obtained by taking all integral domains as models (or, generally, the class of Krull-minimal models~\cite{turkish}) and all partial embeddings as morphisms. This class is not connected: a connected component contains the domains of a fixed characteristic. In the terminology defined in the next section, the rich models of this class are the algebraically closed fields. As a second example, take the class whose models are all infinite structures of signature $L$ and whose morphisms are all partial elementary maps between models. This class is not connected unless $T$ is complete. The connected components consist of models that are elementarily equivalent. The saturated models are the {\it rich\/} models of this class. Finally, highly non trivial examples are obtained from Hrushovski-style constructions such as~\cite{hrushovski}: in such settings, one works with an inductive amalgamation class where models are the models of some theory $T_0$ and morphisms are partial embeddings between self-sufficient subsets.

We conclude this section by stating an important consequence of our axioms: the \emph{finite character} of morphisms, which will be proved in Theorem~\ref{FC}.

\begin{itemize}
\item[\textsf{\footnotesize Fc.}] If all finite restrictions of $f:M\imp N$ are morphisms then $f:M\imp N$ is a morphism.
\end{itemize}

\section{Rich models.}%
\label{rich}

The arguments in this and the following section are either folklore or have appeared in several places e.g.\@ \cite{lascar}, \cite{goode}, \cite{carre}. We fix an inductive amalgamation class $\K$.

\begin{definition} 
Let $\lambda$ be an infinite cardinal. A model $U$ is \emph{$\lambda$--rich\/} if every morphism $f:M\imp U$ such that $|f|<|M|\le\lambda$ has an extension to a strong embedding of $M$ into $U$. That is, there is a total morphism $h:M\imp U$ such that $f\subseteq h$. When $\lambda=|U|$ we say that $U$ is \emph{rich}.
\end{definition}

Using the downward L\"owenheim-Skolem Theorem and FC, it is not difficult to prove that when $\lambda$ is uncountable we can replace $|f|<|M|\le\lambda$ with $|M|<\lambda$ (as in~\cite{chpi}) and obtain an equivalent notion. The case $\lambda=\omega$ does not apply as we do not allow models to be finite.

\begin{exa}\label{randomgraph} 
The countable random graph is a rich model of the inductive amalgamation class which contains all infinite graphs and all partial embeddings between them. All Fra\"{\i}ss\'e limits of finitely generated structures can also be thought of as rich models of a suitably defined inductive amalgamation class. When $\K$ consists of models of some theory $T$ and partial embeddings between them, the $\lambda$-rich models are exactly the existentially closed models of $T$ that are $\lambda$-saturated with respect to quantifier-free types.
\end{exa}

\begin{thm}[Existence]\label{existence} Let $\lambda$ and $\kappa$ be cardinals such that $2^\lambda\le\kappa=\kappa^{<\lambda}$. Then every model $U_0$ of cardinality $\le\kappa$ embeds in a $\lambda$--rich model $U$ of cardinality $\kappa$.
\end{thm}

\begin{proof} Let $U_0$ be given. We may assume $|U_0|=\kappa$. We define by induction a chain of models $\langle U_\alpha:\alpha < \kappa\rangle$ such that $|U_{\alpha}| = \kappa$ for all $\alpha < \kappa$. Let $U:= \bigcup_{\alpha <\kappa}  U_{\alpha}$.

At successor stage $\alpha+1$, let $f:M\imp U_\alpha$ be the least morphism---in a well-ordering that we specify below---such that $|f|<|M|\le\lambda$ and $f$ has no extension to a strong embedding $f':M\imp U_\alpha$. Apply \textsf{\footnotesize Ap} to obtain a strong embedding $f':M\imp U'$ that extends $f:M\imp U_\alpha$. By L\"owenheim-Skolem we may assume $|U_\alpha|=|U'|$. Let $U_{\alpha+1}=U'$. At stage $\alpha$ with $\alpha$ limit, simply let $U_\alpha := \bigcup_{\beta < \alpha}U_\beta$.
\
We choose the required well-ordering so that in the end we forget nobody. At each stage we well-order the isomorphism types of the morphisms $f:M\imp U_\alpha$ such that  $f<|M|\le\lambda$. The required well-ordering is obtained by dovetailing all these well-orderings.  The length of this enumeration is at most $2^\lambda\cdot\kappa^{<\lambda}$, which is $\kappa$ by hypothesis.

We check that $U$ is $\lambda$--rich. Suppose that $f:M\imp U$ is a morphism and $|f|<|M|\le\lambda$. Since
$\kappa^{cf\kappa}>\kappa$ for all $\kappa$, the cofinality of $\kappa$ is larger than $|f|$, hence $\range f\subseteq U_\alpha$ for some $\alpha<\kappa$. So $f:M\imp U_\alpha$ is a morphism and at some stage $\beta$ we have ensured the existence of an extension of $f:M\imp U_\alpha$ that embeds $M$ into $U_{\beta+1}$.
\end{proof}

Theorem~\ref{existence} is too general to yield a sharp bound on the cardinality of $U$. For instance, it cannot be used to infer the existence of countable rich models. However, it will enable us to define $\Tr$ for any inductive amalgamation class.

\begin{cor} 
Let $\lambda$ be an uncountable inaccessible cardinal. Then every model of cardinality $\le\lambda$ embeds in a rich model of cardinality $\lambda$.
\end{cor}

We prefer to work with rich, rather than $\lambda$-rich, models. We assume the existence of as many inaccessible cardinals as needed.

\begin{thm}[Uniqueness]\label{uniqueness} 
Let $U$ and $V$ be $\lambda$--rich models. Then any morphism $f:U\imp V$ is an elementary map. When $|f|<|U|=|V|=\lambda$, $f$ can be extended to an isomorphism.
\end{thm}

\begin{proof} To prove that $f:U\imp V$ is elementary, it suffices to prove that all its finite restrictions are elementary. Therefore we may assume that $f$ itself is finite. Now extend $f$ by back-and-forth to an isomorphism between countable elementary substructures of $U$ and $V$ and the claim is proved. The details are left to the reader.

To prove the second part of the claim, we extend $f:U\imp V$ by back-and-forth, taking care to ensure totality and surjectivity. At limit stages we can safely take unions, since by the first part of the theorem morphisms between $U$ and $V$ are elementary.
\end{proof}

There is a morphism between $U$ and $V$ only if the two models belong to the same connected component. Therefore in each connected component there is at most one rich model of given cardinality.

\begin{cor}[Homogeneity]\label{homogenerity} 
Rich models are homogeneous in the sense that every morphism $f:U\imp U$ of cardinality $<|U|$ has an extension to an automorphism of $U$.
\end{cor}

\begin{thm}[Finite character]\label{FC}  The map $f:M\imp N$ is a morphism if and only if $h:M\imp N$ is a morphism for every finite $h\subseteq f$.
\end{thm}

\begin{proof} One direction is axiom R. For the converse, suppose that for every finite $h\subseteq f$ the map $h:M\imp N$ is a morphism. By Theorem~\ref{existence} we may assume $M,N\le U$ for some rich model $U$. Then $h:U\imp U$ is a morphism and, by Theorem~\ref{uniqueness}, elementary. So $f$ is also elementary on $U$, hence it is a morphism by K2. Since $M,N\le U$, the map $f:M\imp N$ is a morphism because it is a composition of morphisms.
\end{proof}

A \emph{chain of morphisms\/} is a sequence of morphisms $f_\alpha:M_\alpha\imp N_\alpha$, where the $\alpha$--th morphism extends the $\beta$-th morphism for every $\beta<\alpha$. The following is an immediate consequence of the finite character of morphisms.

\begin{cor}\label{unchain} 
The union of a chain of morphisms is a morphism that extends every element of the chain.
\end{cor}

\begin{cor}\label{notAEC} 
Let $\langle M_\alpha : \alpha < \lambda \rangle$ be a chain of models. Let $M_\lambda:= \bigcup_{\alpha < \lambda} M_\alpha$. If $N$ is a model such that $M_\alpha\le N$ for every $\alpha<\lambda$ then $M_\lambda\le N$.
\end{cor}

\begin{proof} 
By \ref{FC} and \ref{unchain}.
\end{proof}

Since $\lambda$--rich models are $\omega$--rich, the following corollary of Theorem~\ref{uniqueness} is immediate.

\begin{cor} 
In each connected component, all rich models have the same theory and this is also the theory of $\lambda$--rich models, for any $\lambda$.
\end{cor}

Let $\Tr$ be the set of sentences that hold in every rich model of the class $\K$. This is called the \emph{theory of the rich models} and it is complete if and only if $\K$ is connected (by Theorem~\ref{uniqueness}).

\section{Saturation}%
\label{saturation}

In this section we show that the saturation of rich models is an intrinsic property of an amalgamation class. This generalizes Proposition~10 in \cite{lascar} or also Theorem 2.5 of~\cite{kuelas}. We also isolate a natural property, which we call {\it fullness}, and show that it does not hold in general (but it holds trivially in all connected amalgamation classes). In the next section, we shall use this property to obtain another characterization of the saturation of rich models.

We fix an inductive amalgamation class $\K$.

\begin{thm} \label{lascar1} Assume that $\K$ is connected. The following are equivalent:
\begin{itemize}
\item[1.] some $\lambda$--rich model is $\lambda$--saturated;
\item[2.] all $\lambda$--rich models are $\lambda$--saturated;
\item[3.] every $\lambda$--saturated model $M\models\Tr$ is $\lambda$--rich.
\end{itemize}
\end{thm}

\begin{proof}  We prove $1\IMP2$. Let $U$ be a $\lambda$--rich and $\lambda$--saturated model. Let $V$ be $\lambda$--rich. We shall use the fact that every morphism between $U$ and $V$, or between elementary substructures of them, is an elementary map. This a consequence of Theorem \ref{uniqueness}. Let $a\in V$ be a tuple of length $<\lambda$. Let $x$ be a finite tuple of variables. We claim that any type $p(x,a)$ is realized in $V$. Let $V'$ be a model of cardinality $\le\lambda$ such that $a\in V'\preceq V$. Since $\K$ is connected there is an elementary embedding $f:V'\imp U$. Let $c$ be such that $U\models p(c,fa)$. Let $U'$ be a model of cardinality $\le\lambda$ such that $fa,c\in U'\preceq U$. Let $h:U'\imp V$ be an elementary embedding that extends $f^{-1}:U'\imp V$. Then $hc$ is the required realisation of $p(x,a)$ in $V$.

To prove $2\IMP3$, assume that $M$ is a $\lambda$--saturated model such that $M\models\Tr$. Let $U$ be a $\lambda$--rich model such that $|U| > |M|$. Let $f:N\imp M$ be a morphism, where $|f|<|N|\le\lambda$. We claim that $f$ can be extended to a strong embedding. Let $M'$ be a structure of cardinality $\le\lambda$ such that  $\range f \subseteq M'\preceq M$. As $\Tr$ is a complete theory, $U\equiv M'$ and, by $\lambda$--saturation, there is an elementary embedding $g:M'\imp U$. By $\lambda$--richness, there is a morphism $h:N\imp U$ that extends $gf:N\imp U$. As $M$ is $\lambda$--saturated, there is an elementary embedding $k:h[N]\imp M$. Then $k:U\imp M$ is a morphism, so $kh:N\imp M$ is the required embedding.

Finally, the implication $3\IMP1$ is clear.\end{proof}

An analogous theorem holds for saturated rich models. The proof is similar.

\begin{thm}\label{oneallsaturated} Assume that $\K$ is connected. The following are equivalent:
\begin{itemize}
\item[1.]   some rich model is saturated;
\item[2.]   all rich models are saturated;
\item[3.]   every saturated model $M\models\Tr$ is rich.
\end{itemize}
\end{thm}

When $\K$ is not connected these results hold within each connected component.


\begin{thm}\label{allsaturatedmc} Let $\lambda$ be any infinite cardinal. The following are equivalent:
\begin{itemize}
\item[1.]  all $\lambda$--rich models are $\lambda$--saturated;
\item[2.]  all rich models are saturated;
\item[3.]  if $U$ is rich, $M\equiv U$, and $M\le U$, then $M\preceq U$;
\item[4.]  if $U$ is rich, $M\equiv U$,  then any morphism $f:M\imp U$ is elementary.
\end{itemize}
\end{thm}

\begin{proof} The equivalence $3\IFF4$ is clear. We prove $1 \IMP 3$. Suppose that $U$ is rich. We may assume that $\lambda\le|U|$ (otherwise we prove the claim for a sufficiently large rich model in the same connected component as $U$; then 3 follows easily). By 1, $U$ is saturated. Let $A\subseteq M$ be any finite set and let $M'$ be a countable model such that $A\subseteq M'\preceq M$. If we show that $M'\preceq U$, $M\preceq U$ follows from the arbitrariness of $A$. As $M'\equiv U$, by saturation there is a model $M''\preceq U$ which is isomorphic to $M'$. Let $f:M'\imp M''$ be this isomorphism. Then $f:U\imp U$ is a morphism and, as $U$ is rich, an elementary map by~\ref{uniqueness}. So $M'\preceq U$ as required. The implication $2 \IMP 3$ is similar.

Finally, we assume 4 and prove that if $U$ is $\lambda$--rich then it is $\lambda$--saturated. As $\lambda$ is arbitrary, both $4\IMP1$ and $4\IMP 2$ follow. Let $p(x)$ be a type over some set $A\subseteq U$ of cardinality $<\lambda$. Fix some model $M\equiv_A U$ of cardinality $\le \lambda$ that realizes $p(x)$. By 4, there is an elementary embedding $f:M\imp U$ over $A$. Hence $U$ realizes $p(x)$.
\end{proof}

\begin{cor}\label{morphismselementary} Let $U$ be a rich saturated model. Then for any $M\equiv N\equiv U$, every morphism $f:M\imp N$ is elementary.
\end{cor}

\begin{proof} Let $V$ be a rich model and let $h:N\imp V$ be a strong embedding. Since $V$ and $U$ are in the same connected component, they are elementarily equivalent. Then $h:N\imp V$ and $hf:M\imp V$ are elementary by Theorem~\ref{allsaturatedmc}. It follows that $f:M\imp N$ is elementary.
\end{proof}

\begin{exa}\label{automorphismsrandomgraph}
Truss-generic automorphisms of the random graph.  Let $L$ be the language of graphs and let $T$ be the theory of the random graph. Let $L_0\smallsetminus L$ contain two unary function symbols $f$ and $f^{-1}$ and let $T_0$ be $T$ together with a sentence which says that $f$ is an automorphism with inverse $f^{-1}$. The morphisms of $\K$ are partial embeddings. It is not difficult to verify that the class $\K$ axiomatized by $T_0$ has the amalgamation property and is in fact an inductive amalgamation class. It is known~\cite{kikyo} that $T_0$ has no model companion, hence rich models are not saturated.
\end{exa}

\begin{exa}\label{autnocyles} Cycle-free automorphisms of the random graph. Let $L$, $T$, $N$, and $L_0$ be as in Example~\ref{automorphismsrandomgraph}.  The theory $T_0$ says that $f$ is an automorphism with inverse $f^{-1}$, and moreover for every positive integer $n$ it contains  the axiom $\A x\,f^nx\neq x$. These axioms claim that $f$ has no finite cycles. It is not difficult to verify that the class $\K$ axiomatized by $T_0$ has the amalgamation property and is an inductive amalgamation class if morphisms are partial isomorphisms between models. It is known~\cite{kumac} that $T_0$ has a model-companion, hence rich models are saturated.\end{exa}

\begin{exa}\label{blackfields} Poizat's black fields, uncollapsed version. This is a paradigmatic example among many possible versions of Hrushovski's amalgamation constructions. We refer to~\cite{carre} for all unproved claims. Let $L$ be the language of rings and $T$ the theory of algebracally closed fields of a given characteristic. Let $L_0$ contain a unary predicate $r$. Define 
$$
\delta(A)\ \ =\ \ 2\cdot\deg(A) - |r(A)|,
$$
where $\deg(A)$ is the trascendence degree of $A$. Define a universal theory $T_0$ translating into first-order sentences the requirement that $0\le\delta(A)$ holds for every finite set $A$. 

Fix $(M,\sigma)\models T_0$ and let $A\subseteq M$. We write $A\sqsubseteq M$, if for every finite $B\subseteq M$ we have that $\delta(A\cap B)\le \delta(B)$. Let $\acl(A)$ denote the algebraic closure of $A$ in the signature $L$. Observe that, as $T$ is a complete theory, this does not depend on $M$ nor on $\sigma$. We say that $A$ is a \emph{strong subset\/} of $M$ if $\acl(A)=A\sqsubseteq M$. We write $\textrm{cl}(A)$ for the intersection of all strong subsets of $M$ containing $A$. This is called the \emph{closure\/} of $A$; it clearly depends on $\sigma$ though we do not display it in the notation. It is not difficult to prove that $\textrm{cl}(A)$ is a strong subset.

The morphisms of $\K$ are the maps $f:(N,\tau)\imp (M,\sigma)$ that have an extension to a partial isomorphim $h:(N,\tau)\imp (M,\sigma)$ with $\dom H$ and $\range H$ self-sufficient in $(N,\tau)$ and $(M,\sigma)$ respectively. It is easy to show that if $(N,\tau)\preceq_1(M,\sigma)\models T_0$, then $N$ is self-sufficient in $(M,\sigma)$. So axiom K2 holds. Axiom AP is easily verified by free amalgamation and all the other axioms are clear.


No element of $\acl(\varnothing)$ satisfies $r(x)$ so the class is connected. Rich models are saturated (this uses the definability of Morley rank in algebraically closed fields).
\end{exa}

The models $M$ and $N$ in Theorem~\ref{allsaturatedmc} and its corollaries are required to be elementarily equivalent to some rich model. It would be convenient to replace this condition by $M, N\models\Tr$ but this is not possible in general: the following example shows that there may be models where $\Tr$ holds which are not elementarily equivalent to any rich model.

\begin{exa}\label{counterexample} The language $L_0$ contains a binary predicate $r$ and the constants $c_n$, for $n\le\omega$. Consider the structures of signature $L_0$ where the following axioms hold:
\begin{itemize}
\item[0.] $c_i\neq c_j$ for every distinct $i,j\le\omega$,
\item[1.] $\A x\ \neg r(x,x)$,
\item[2.] $\A x\,y\ [r(x,y)\leftrightarrow r(y,x)]$,
\item[3.] $\E x\, r(c_i,x)\ \imp\ \neg \E x\, r(c_j,x)$ for every distinct $i,j\le\omega$.
\end{itemize}
These are graphs with countably many vertices named. The named vertices are, with one possible exception, isolated. The inductive amalgamation class $\K$ is the disjoint union of the classes $\K_n$ defined as follows for $n\le\omega$. For $n<\omega$, the models of $\K_n$ are the graphs that satisfy Axioms 0--3 above and
\begin{itemize}
\item[a.]  $\E x\; r(c_n,x)$, or 
\item[b.]  $\neg \E x\; r(c_i,x)$ for every $i\le\omega$ and there are exactly $n$ triangles (i.e.\@ cliques of size $3$).
\end{itemize}
The models of $\K_\omega$ satisfy Axioms 0--3 above and
\begin{itemize}
\item[a$'$.]  $\E x\; r(c_\omega,x)$, or 
\item[b$'$.]  $\neg \E x\; r(c_k,x)$ and there are more than $k$ triangles for every $k<\omega$
\end{itemize}
Each $\K_n$ contains two sorts of graphs: those where $c_n$ is the unique constant which is non-isolated and those where all constants are isolated. When  all the constants are isolated, the graph contains exactly $n$ triangles if $n<\omega$, or infinitely many if $n=\omega$. 

The morphisms of $\K_n$ are the partial embeddings. In $\K$ there is no other morphism than those between models in the same component $\K_n$. It is easy to see that $\K$ is an inductive amalgamation class. Since models in different components are not elementarily equivalent \textsf{\footnotesize K1} holds. To prove \textsf{\footnotesize Ap} it suffices to show that if $M_1$ and $M_2$ are models in the same component $\K_n$ and $M_1\cap M_2$ is a common substructure, then there is a model $N$ that is a superstructure of both $M_1$ and $M_2$. There are two cases. If $M_i\models\E x\, r(c_n,x)$ for either one of $i\in\{1,2\}$, we let $N$ be the free amalgam of $M_1$ and $M_2$ over $M_1\cap M_2$, that is, $N=M_1\cap M_2$ with no extra edges added. Otherwise we take $N=M_1\cup M_2\cup\{a\}$, were $a$ is a new vertex and let $r^N:=r^{M_1} \cup r^{M_2}\cup\{\langle c_n,a\rangle, \langle a,c_n\rangle\}$. Axioms 0--3 clearly hold in $N$.

We now describe a countable rich model $U\in\K_n$. This is the disjoint union of two structures $U_{\rm rand}$ and $U_{\rm isol}$: the first is a random graph, and the second contains only isolated vertices. The structure $U_{\rm rand}$ contains $c_n$, while $U_{\rm isol}$ contains all other constants and infinitely many other vertices.

The model $U$ is rich. Let $f:M\imp U$ be a morphism, with $|f|<|M|\le |U|$. We can extend $f$ to $f'$ so that $\{c_i : i\le\omega\}\subseteq\dom f'$. Let $f'=f_{\rm rand}\cup f_{\rm isol}$ where $\range f_{\rm rand}\subseteq U_{\rm rand}$ and $\range f_{\rm isol}\subseteq U_{\rm isol}$. We can extend $f_{\rm rand}$ to an embedding of $M\smallsetminus\dom f_{\rm isol}$ into $U_{\rm rand}$, because this is a random graph. This proves that $U$ is rich.

Consider a structure $M$ which is the disjoint union of a countable random graph and a set of isolated vertices containing all the constants and infinitely many other elements. Since in $M$ all constants are isolated, $M$ is not elementary equivalent to any rich model. But every formula $\phi$ true in $M$ also holds in some rich model $U$ (e.g.\@ if $c_n$ does not occur in $\phi$, then $\phi$ will hold in $U\in\K_n$).
\end{exa}

The example above motivates the following definition.

\begin{definition}\label{deffull} 
An inductive amalgamation class is \emph{full\/} if for every model $M$ the following holds: if each sentence true in $M$ is also true in some rich model $U_\phi$ then some rich model $U$ satisfies $\Th(M)$. Equivalently, $\K$ is full if in each connected component only one completion of $\Tr$ is realized by a model.
\end{definition}

The following theorem generalizes Theorems~\ref{lascar1} and~\ref{allsaturatedmc}.

\begin{thm}\label{sumupsaturation} Suppose $\K$ is full. Then the following are equivalent:
\begin{itemize}
\item[1.] all rich models are saturated;
\item[2.] all $\lambda$--rich models are $\lambda$--saturated;
\item[3.] all saturated model $M\models\Tr$ are rich;
\item[4.] all morphisms between models $M,N\models\Tr$ are elementary;
\item[5.] $M\le N\ \IFF\ M\preceq N$, for any pair of models $M,N\models\Tr$.
\end{itemize}
\end{thm}

\section{Model companions}\label{companions}

In this section we review some results of~\cite{chpi}, namely Section~3.4 and Proposition~3.5 and we show that they hold in the context of inductive amalgamation classes. We also prove that the existence of model companions is equivalent to fullness of the class plus saturation of rich models. 

We will work under the following condition

\begin{itemize} 
\item[\#] If $M, N \models \Tr$ are models of $\K$,  then $M\subseteq N\ \IFF\ M\le N$.
\end{itemize}

This is equivalent to requiring that any embedding $f:M\imp N$ between $M, N \models \Tr$ is strong, i.e.\@ a morphism. In fact, as $M$ is isomorphic to $f[M]$, then $f[M]$ is in $\K$ and entails $\Tr$, so \# implies that $f[M]\le N$. Then $f:M\imp N$ is the composition of two morphisms, hence a morphism.

\begin{thm}\label{mc-saturated} Assume that \# holds in $\K$. Then the following are equivalent:
\begin{itemize}
\item[1.] $\Tr$ is model-complete;
\item[2.] all rich models are saturated and $\K$ is full.
\end{itemize}
\end{thm}

\begin{proof} By \# we can replace `$\le$' with `$\subseteq$' in the last assertion of Theorem~\ref{sumupsaturation} and obtain 
\begin{itemize} 
\item[$\dagger$]  if $M,N\models\Tr$, then $M\subseteq N\ \IFF\ M\preceq N$.
\end{itemize}
Observe that $\dagger$ implies that $\K$ is full.
\end{proof}

We say that $\K$ is \emph{axiomatizable\/} if there is a theory $T_0$ such that $M$ is a model if and only if $M \models T_0 $. In this case, we also say that $\K$ \emph{is axiomatised by\/} $T_0$.

\begin{thm}\label{universalconsequences} 
Assume that $\K$ is axiomatised by a theory $T_0$. Then $T_{0,\A}={\Tr}_{,\A}$.
\end{thm}

\begin{proof} 
Clearly $T_0\subseteq\Tr$. Since every structure modelling $T_0$ is a model, it is a substructure of a rich model. Therefore ${\Tr}_{,\A}\subseteq T_{0,\A}$.
\end{proof}

\begin{thm}\label{mcsat}
Assume that $\K$ is axiomatized by $T_0$ and that \# holds in $\K$. Then the following are equivalent:
\begin{itemize}
\item[1.] $\Tr$ is model-complete;
\item[2.] $\Tr$ is the model companion of $T_0$;
\item[3.] all rich models are saturated and $\K$ is full.
\end{itemize}
Conversely, if $T_0$ has a model companion, then $\Tr$ is this model companion.
\end{thm}

\begin{proof} The equivalences $1\IFF2\IFF3$ are clear by Theorems~\ref{mc-saturated} and~\ref{universalconsequences}. To prove the second claim, we assume $T_0$ has a model companion $T_c$. To see that $T_c\subseteq\Tr$ it suffices to observe that, by $\#$, rich models are existentially closed, so $T_c$ holds in every rich model. To prove the converse inclusion, let $M_0\models T_c$ be any structure. We claim that $M_0\models\Tr$. As $T_{0,\A}={\Tr}_{,\A}$, every structure $M\models T_c$ is a substructure of a rich model. Conversely, every rich model is a substructure of some $M\models T_c$, so we can construct a chain of substructures
$$
M_0\subseteq U_0\subseteq M_1\subseteq U_1\subseteq M_2\subseteq\dots\dots,
$$
where $M_i\models T_c$ and $U_i$ is a rich model. It follows that $M_i\preceq M_{i+1}$ and $U_i\preceq U_{i+1}$. Let
$$
U_\omega\ :=\ \bigcup_{i\in\omega}U_i\ =\ \bigcup_{i\in\omega}M_i.
$$ 
Then $M_0\preceq U_\omega$. The union of a chain of rich models is $\omega$--rich, so the theorem follows.\end{proof}

\begin{rem}\label{fullnecessary} The requirement of fullness in 3 of Theorem~\ref{mcsat} is necessary. All rich models in Example~\ref{counterexample} are saturated, but $\Tr$ is not model-complete: the formula $\E y\,r(x,y)$ is not equivalent over $\Tr$ to any universal formula. In fact $\E y\,r(x,y)$ is not preserved under substructure: if $U$ is a rich model in $\K_{\omega}$ then $U \models\E y\,r(c_{\omega},y)$, but in the model $M \subseteq U$ constructed at the end of Example~\ref{counterexample} we have $\neg\E y\,r(c_{\omega},y)$.\end{rem}

\begin{exa}\label{excompletesmall} Let $T$ be any complete small theory with quantifier elimination in the language $L$. Let $L_0\smallsetminus L$ contain only a unary relation symbol $r$ and let $T_0=T$. We define an inductive amalgamation class $\K$. The models of $\K$ are the structures that model $T_0$. The morphisms of $\K$ are partial isomorphisms that have a  domain of definition which is algebraically closed in $T$, as well as any restriction of these maps. It is easy to verify that all the axioms of Section~\ref{definitions} hold in $\K$  (free amalgamation suffices to prove AP). Hypothesis \# is trivially satisfied.

Let $\acl(A)$ denote the algebraic closure in $T$. If $\acl(\varnothing)\neq\varnothing$ the class is not connected: the set $\{a\in\acl(\varnothing) : (M,\sigma)\models r(a)\}$ determines the connected component of the model $(M,\sigma)$. In~\cite{chpi} it is proved that if $T$ eliminates the $\E^{\infty}$ quantifier, then $T_0$ has a model companion: $\Tr$.\end{exa}

\begin{exa}\label{PAPA} Let $T$ and $L$ be as in Example \ref{excompletesmall}. Let $L_0\smallsetminus L$ contain two unary function symbols $f$ and $f^{-1}$ and let $T_0$ be $T$ together with a sentence which says that $f$ is an automorphism with inverse $f^{-1}$. The class $\K$ is defined as in Example \ref{excompletesmall}. Here the amalgamation property is not trivial: when it holds one says that $T$ has the PAPA \cite{lascar}. So suppose $T$ has the PAPA. Then hypothesis \# is again trivially satisfied.

This class is not connected. As in the examples above, the restriction of $f$ to $\acl(\varnothing)$ determines the connected component of $\K$ to which the model belongs. It is considerably more difficult to find a condition which guarantees the model-completeness of $\Tr$ ~\cite{balshel}. An important example is the case where $T$ is the theory of algebraically closed fields, then $\Tr$ is also kown as ACFA. Let $N$ be a countable algebraically closed field of infinite transcendence degree.\end{exa}

\end{document}